\documentclass[12pt]{amsart}
\usepackage{amssymb,color}
\numberwithin{equation}{section} 

\usepackage{graphicx}

\newcommand{\bea}{\begin{eqnarray}}
\newcommand{\eea}{\end{eqnarray}}
\newcommand{\ba}{\begin{array}}
\newcommand{\ea}{\end{array}}
\newcommand{\edc}{\end{document}}
\newcommand{\bc}{\begin{center}}
\newcommand{\ec}{\end{center}}
\newcommand{\be}{\begin{equation}}
\newcommand{\ee}{\end{equation}}

\newcommand{\dsf}{\displaystyle\frac}

\def\cb{{\mathcal B}}

\def\ce{{\mathcal E}}
\def\cf{{\mathcal F}}
\def\cg{{\mathcal G}}

\def\cn{{\mathcal N}}

\def\bc{{\mathbb C}}

\def\bn{{\mathbb N}}

\def\bq{{\mathbb Q}}
\def\br{{\mathbb R}}

\def\bz{{\mathbb Z}}

\def\g{\gamma}  \def\G{\Gamma}
  \def\D{\Delta}

\def\l{\lambda} 

\def\m{\mu}

\def\n{\nu}

\def\s{\sigma}

\def\w{\omega} \def\Om{\Omega}

\def\h{{\mathbf{h}}}

\def\sb{{\mathbf{s}}}

\newtheorem{thm}{Theorem}[section]
\newtheorem{lem}[thm]{Lemma}
\newtheorem{cor}[thm]{Corollary}
\newtheorem{conj}[thm]{Conjecture}
\newtheorem{prop}[thm]{Proposition}

\theoremstyle{remark}
\newtheorem{rem}{Remark}[section]

\begin{document}

\title[$p$-adic Ising-Vannimenus model]
{On $P$-adic Ising-Vannimenus model on an arbitrary order Cayley
tree}


\author{Farrukh Mukhamedov}
\address{Farrukh Mukhamedov\\
 Department of Computational \& Theoretical Sciences\\
Faculty of Science, International Islamic University Malaysia\\
P.O. Box, 141, 25710, Kuantan\\
Pahang, Malaysia} \email{{\tt far75m@yandex.ru} {\tt farrukh\_m@iium.edu.my}}
\author{ Mansoor Saburov}
\address{ Mansoor Saburov\\
Department of Computational \& Theoretical Sciences\\
Faculty of Science, International Islamic University Malaysia\\
P.O. Box, 141, 25710, Kuantan\\
Pahang, Malaysia} \email{{\tt msaburovm@gmail.com}}
\author{ Otabek Khakimov}
\address{ Otabek Khakimov\\
Institute of mathematics, 29, Do'rmon Yo'li str., 100125,
Tashkent, Uzbekistan.} \email {{\tt hakimovo@mail.ru}}

\begin{abstract}
In this paper, we continue an investigation of the $p$-adic
Ising-Vannimenus model on the Cayley tree of an arbitrary order
$k$ $(k\geq 2$). We prove the existence of $p$-adic quasi Gibbs
measures by analyzing fixed points of multi-dimensional $p$-adic
system of equations. We are also able to show the uniqueness of
translation-invariant $p$-adic Gibbs measure. Finally, it is
established the existence of the phase transition for the
Ising-Vannimenus model depending on the order $k$ of the Cayley
tree and the prime $p$.
 Note that the methods used in the
paper are not valid in the real setting, since all of them are based
on $p$-adic analysis and $p$-adic probability measures.

\end{abstract}

\maketitle

\section{introduction}

At the first place, J. Vanniminus started to study the Ising model
with nearest-neighbor and next-nearest-neighbor interactions on
the Cayley tree of order two in the paper \cite{V}. The similar
results were numerically obtained for an arbitrary order Cayley
tree in \cite{GATU}.
 The $p$-adic counterpart of the
Ising-Vanniminus model on the Cayley tree of order two was first
studied in  \cite{MFDMAH}. There, it was proposed a
measure-theoretical approach to investigate the model in the
$p$-adic setting. The proposed methods have been  based on
$p$-adic probability measures. In this paper we consider the
Ising-Vanniminus model on an arbitrary order Cayley tree in the
$p$-adic setting \footnote{Note that $p$-adic numbers provide a
more exact and more adequate description of microworld phenomena.
Therefore, there are many papers are devoted to the description of
various models in the language of $p$-adic analysis (see for
example
\cite{AKS,ADV,ADFV},\cite{FO},\cite{Kh1,Kh2,KK2,KK3},\cite{MP},\cite{V1,V2,VVZ}).}.
Note that the Cayley tree or Bethe lattices (see \cite{Ost}) were
fruitfully used, providing a deeper insight into the behavior of
the models. Moreover, they will provide more information about the
models defined on complex networks \cite{DGM}.

In the present paper we use the methods based on $p$-adic
probability measures. We point out that the $p$-adic numbers
appeared in quantum physics models (see \cite{BC}) such a way the
model is described by $p$-adic probability measures. There are
also many models in physics cannot be described using ordinary
Kolmogorov's probability theory (see \cite{Kh2,Ko,MP,VVZ}). These
papers stimulated the development of the $p$-adic probability
models \cite{BD},\cite{Kh1,K3,KYR,Ro}. Moreover, an abstract
non-Archimedean measure theory developed in \cite{KL,Lu} which
laid on the base of the theory of stochastic processes with values
in $p$-adic and more general non-Archimedean fields. Hence, this
theory of stochastic processes allowed us to construct wide
classes of processes using finite dimensional probability
distributions. Therefore, in \cite{GRR,M13,M15,MR1,RKhak} it has
been developed a $p$-adic statistical mechanics models with
nearest neighbor interactions, based on the theory of $p$-adic
probability measures and processes.  Namely, we have studied
$p$-adic Ising and Potts models with nearest neighbor interactions
on Cayley trees. Note that there are also several $p$-adic models
of complex hierarchic systems \cite{KK2,KK3}. We remark that one
of the central problems of such a theory is the study of
infinite-volume Gibbs measures corresponding to a given
Hamiltonian, and a description of the set of such measures. In
most cases such an analysis depend on a specific properties of
Hamiltonian, and complete description is often a difficult
problem. This problem, in particular, relates to a phase
transition of the model (see \cite{G}).

In this paper, we consider a $p$-adic analogue of the model
\cite{V} on the Cayley tree, i.e. such a model has nearest
neighbor and next nearest neighbor interactions. Here, we continue
an investigation of the $p$-adic Ising-Vanniminus model on the
Cayley tree of an arbitrary order. We are going to use a
measure-theoretical approach proposed in \cite{MFDMAH}. In the
previous study, the order of the Cayley tree being two was
essentially useful. In this study, we are going to apply different
techniques which will allow us to prove the uniqueness of
translation-invariant $p$-adic Gibbs measure. This partially
confirms the conjecture formulated in \cite{MFDMAH} to be true.
Moreover, we will also show the existence of the phase transition
depending on the order $k$ of the Cayley tree and prime $p$.
 We point out
that in the $p$-adic setting there are several kinds of phase
transitions such as \textit{strong phase transition},
\textit{phase transition} (see \cite{M13,M14}). Here, by the
\textit{phase transition} we mean the existence of at least two
non-trivial $p$-adic quasi Gibbs measures such that one is bounded
and the second one is unbounded (note that in the $p$-adic
probability, unlike to real setting, the probability measures
could be even unbounded \cite{Ro}). The reader should refer to
\cite{DKKV} for the recent development of the subject.

The present paper is organized as follows.  In section 2, we
recall some necessary results from the $p$-adic analysis. In
Section 3, we recall definitions of the Ising-Vanniminus model and
corresponding $p$-adic quasi Gibbs measures via interacting
functions. Note that such kind of measures exist if the
interacting functions satisfy multi-dimensional recurrence
equations. In section 4, the existence of $p$-adic quasi Gibbs
measures is established in terms of the order of the tree $k$ and
the prime $p$. To prove such an existence result we are going to
investigate fixed points of a multi-dimensional functional
equation. In Section 5, we will show that the
translation-invariant $p$-adic Gibbs measure is unique. In
particularly, the obtained result confirms the conjecture
formulated in \cite{MFDMAH} to be true. In the final section, we
are able to establish the existence of the phase transition for
the model depending on the order $k$ of the Cayley tree and the
prime $p$.  Note that the values of the norms (i.e. "absolute
values" ), in the $p$-adic setting, are discrete, therefore it is
impossible to determine exact values of a critical point for the
existence of the phase transition. Moreover, in the field of
$p$-adic numbers there is no reasonable order compatible with the
usual order in the rational numbers, and this rises some
difficulties in the direction of determination of exact critical
points. We stress that the methods used in the paper are not valid
in the real setting, since all of them are based on $p$-adic
analysis and $p$-adic probability measures.

\section{Preliminaries}

\subsection{$p$-adic numbers}

In what follows $p$ will be a fixed prime number. The set $\bq_p$ is
defined as a completion of the rational numbers $\bq$ with respect
to the norm $|\cdot|_p:\bq\to\br$ given by
\begin{eqnarray}
|x|_p=\left\{
\begin{array}{c}
  p^{-r} \ x\neq 0,\\
  0,\ \quad x=0,
\end{array}
\right.
\end{eqnarray}
here, $x=p^r\frac{m}{n}$ with $r,m\in\bz,$ $n\in\bn$,
$(m,p)=(n,p)=1$. The absolute value $|\cdot|_p$ is non-Archimedean,
meaning that it satisfies the strong triangle inequality $|x + y|_p
\leq \max\{|x|_p, |y|_p\}$. We recall a nice property of the norm,
i.e. if $|x|_p>|y|_p$ then $|x+y|_p=|x|_p$. Note that this is a
crucial property which is proper to the non-Archimedenity of the
norm.

Any $p$-adic number $x\in\bq_p$, $x\neq 0$ can be uniquely represented in the form
\begin{equation}\label{canonic}
x=p^{\g(x)}(x_0+x_1p+x_2p^2+...),
\end{equation}
where $\g=\g(x)\in\bz$ and $x_j$ are integers, $0\leq x_j\leq p-1$,
$x_0>0$, $j=0,1,2,\dots$ In this case $|x|_p=p^{-\g(x)}$.

We recall that an integer $a\in \bz$ is called {\it the $k^{th}$
residue modulo $p$} if the congruent equation $x^k\equiv a(\operatorname{mod })p$
has a solution $x\in \bz$.

By $\mathbb F_p=\bz/p\bz$ we denote a subgroup of $\bq_p$. Recall
that  $\sqrt[k]{-1}$ exists in $\mathbb F_p$ whenever $-1$ is the
$k^{th}$ residue of module $p$. Otherwise, it is said  that
$\sqrt[k]{-1}$ does not exist in $\mathbb F_p$. Let
$\cn_{k,p}(\mathbb{F}_p)$ be the number of solutions
$x^k\equiv-1(\operatorname{mod }p)$ in $\mathbb F_p$. It is known
\cite{Rosen} that  $\sqrt[k]{-1}$ exists in $\mathbb F_p$ if and
only if $\frac{p-1}{(k,p-1)}$ is even. Moreover, in this case, one
has that $\cn_{k,p}(\mathbb{F}_p)=(k,p-1)$. Similarly, we say
that $\sqrt[k]{-1}$ exists in $\mathbb Q_p$ whenever the equation
$x^k=-1$ is solvable in $\mathbb Q_p$. Otherwise, it is said  that $\sqrt[k]{-1}$
does not exist in $\mathbb Q_p$. It was shown \cite{MS} that
$\sqrt[k]{-1}$ exists in $\mathbb Q_p$ if and only if
$\sqrt[q]{-1}$ exists in $\mathbb F_p$, where $k=q\cdot p^s$ and
$(q,p)=1$ with $s\geq 0.$

For each $a\in \bq_p$, $r>0$ we denote $$ B(a,r)=\{x\in \bq_p :
|x-a|_p< r\}$$
and the set of all {\it $p$-adic integers}
$$\bz_{p}=\left\{ x\in \bq_{p}:\
|x|_{p}\leq1\right\}.$$
The set $\bz_p^*=\bz_p\setminus p\bz_p$ is called a set of $p$-adic units.

The following lemma is known as the Hensel's lemma
\begin{lem}\cite{AKP,I}\label{hl}
Let $\mathbf{x}=(x_1,\cdots,x_m),$ $\Theta=(0,\cdots,0)$ and $F(\mathbf{x})=(f_1(\mathbf{x}),\cdots, f_m(\mathbf{x}))$ be a polynomial function whose coefficients are $p$-adic integers. Let $\mathbb{J}(F(\mathbf{x}))$ be the Jacobian matrix of the function $F(\mathbf{x})$.
If there exists a vector $\mathbf{a}=(a_1,\cdots,a_m)$ with $p$-adic integer components such that
$$
F(\mathbf{a})\equiv \Theta \ ({mod} \ p)\ \ \mbox{and}\ \ \ det(\mathbb{J}(F(\mathbf{a})) \not\equiv0 \ (mod \ p)$$
then $F(\mathbf{x})$ has a unique root $\mathbf{x}_0$ with $p$-adic integer components which satisfies $\mathbf{x}_0\equiv \mathbf{a} \ (mod \ p)$.
\end{lem}

Recall that the $p$-adic logarithm is defined by the series
$$
\log_p(x)=\log_p(1+(x-1))=\sum_{n=1}^{\infty}(-1)^{n+1}\dsf{(x-1)^n}{n},
$$
which converges for every $x\in B(1,1)$. And the $p$-adic
exponential is defined by
$$
\exp_p(x)=\sum_{n=0}^{\infty}\dsf{x^n}{n!},
$$
which converges for every $x\in B(0,p^{-1/(p-1)})$.

\begin{lem}\label{21} \cite{Ko},\cite{VVZ} Let $x\in
B(0,p^{-1/(p-1)})$ then we have $$ |\exp_p(x)|_p=1,\ \ \
|\exp_p(x)-1|_p=|x|_p<1, \ \ |\log_p(1+x)|_p=|x|_p<p^{-1/(p-1)} $$
and $$ \log_p(\exp_p(x))=x, \ \ \exp_p(\log_p(1+x))=1+x. $$
\end{lem}

In what follows, we will use the following auxiliary facts.

\begin{lem}\cite{KMM}\label{pr} If $|a_i|_p\leq 1$, $|b_i|_p\leq 1$, $i=1,\dots,n$, then
\begin{equation*}
\bigg|\prod_{i=1}^{n}a_i-\prod_{i=1}^n b_i\bigg|_p\leq \max_{i\leq
i\leq n}\{|a_i-b_i|_p\}
\end{equation*}
\end{lem}

Put
$$
\ce_p=\{x\in\bq_p: \ |x|_p=1, \ \ |x-1|_p<p^{-1/(p-1)}\}.
$$
As corollary of Lemma \ref{21} we have the following
\begin{lem}\label{epproperty}
The set $\ce_p$ has the following properties:\\
$(a)$ $\ce_p$ is a group under multiplication;\\
$(b)$ $|a-b|_p<1$ for all $a,b\in\ce_p$;\\
$(c)$ If $a,b\in\ce_p$ then it holds
\[
|a+b|_p=\left\{\begin{array}{ll}
\frac{1}{2}, & \mbox{if }\ p=2\\
1, & \mbox{if }\ p\neq2.
\end{array}\right.
\]\\
$(d)$ If $a\in\ce_p$, then
there is an element $h\in B(0,p^{-1/(p-1)})$ such that
$a=\exp_p(h)$.
\end{lem}

Note that the basics of $p$-adic analysis, $p$-adic mathematical
physics are explained in \cite{Ko,R,VVZ}.

\subsection{$p$-adic measure}

Let $(X,\cb)$ be a measurable space, where $\cb$ is an algebra of
subsets $X$. A function $\m:\cb\to \bq_p$ is said to be a {\it
$p$-adic measure} if for any $A_1,\dots,A_n\subset\cb$ such that
$A_i\cap A_j=\emptyset$ ($i\neq j$) the equality holds
$$
\mu\bigg(\bigcup_{j=1}^{n} A_j\bigg)=\sum_{j=1}^{n}\mu(A_j).
$$

A $p$-adic measure is called a {\it probability measure} if
$\mu(X)=1$.  One of the important condition (which was already
invented in the first Monna--Springer theory of non-Archimedean
integration \cite{Mona}) is boundedness, namely a $p$-adic
probability measure $\m$ is called {\it bounded} if
$\sup\{|\m(A)|_p : A\in \cb\}<\infty $. We pay attention to an
important special case in which boundedness condition by itself
provides a fruitful integration theory (see for example
\cite{Kh07}). Note that, in general, a $p$-adic probability
measure need not be bounded \cite{K3,KL,Ko}. For more detail
information about $p$-adic measures we refer to \cite{AKh,K3}.

\subsection{Cayley tree}

Let $\Gamma^k_+ = (V,L)$ be a semi-infinite Cayley tree of order
$k\geq 1$ with the root $x^0$ (whose each vertex has exactly $k+1$
edges, except for the root $x^0$, which has $k$ edges). Here $V$ is
the set of vertices and $L$ is the set of edges. The vertices $x$
and $y$ are called {\it nearest neighbors} and they are denoted by
$l=\langle{x,y}\rangle$ if there exists an edge connecting them. A collection of
the pairs $\langle{x,x_1}\rangle,\dots,\langle{x_{d-1},y}\rangle$ is called a {\it path} from
the point $x$ to the point $y$. The distance $d(x,y), x,y\in V$, on
the Cayley tree, is the length of the shortest path from $x$ to $y$.
$$
W_{n}=\left\{ x\in V\mid d(x,x^{0})=n\right\}, \ \
V_n=\overset{n}{\underset{m=1}{\bigcup}}W_{m}, \ \ L_{n}=\left\{
l=<x,y>\in L\mid x,y\in V_{n}\right\}.
$$
The set of direct successors of $x$ is defined by
$$
S(x)=\left\{ y\in W_{n+1}:d(x,y)=1\right\}, x\in W_{n}.
$$
Observe that any vertex $x\neq x^{0}$ has $k$ direct successors and
$x^{0}$ has $k+1$.

Two vertices $x,y\in V$ are called the \textit{next-nearest
neighbors} if $d(x,y)=2$. The next-nearest-neighbors vertices $x$
and $y$ are called the \textit{prolonged next-nearest neighbors}
if $x\in W_{n-2}$ and $y\in W_n$ for some $n\geq 1$, which are
denoted by $\rangle{x,y}\langle$.

Now we are going to introduce a coordinate structure in $\G_+^k$.
Every vertex $x$ (except for $x^{0}$) of $\G_+^k$ has coordinates
$(i_1,\dots,i_n)$, here $i_m\in\{1,\dots,k\},\ 1\leq m\leq n$ and
for the vertex $x^0$ we put $(0)$ (see Figure 1). Namely, the
symbol $(0)$ constitutes level $0$ and the sites $i_1,\dots,i_n$
form level $n$ of the lattice. In this notation for $x\in\G_+^k,\
x=\{i_1,\dots,i_n\}$ we have
$$
S(x)=\{(x,i): 1\leq i\leq k\},
$$
here $(x,i)$ means that $(i_1,\dots,i_n,i)$.

\begin{figure}
\begin{center}
\includegraphics[width=10.07cm]{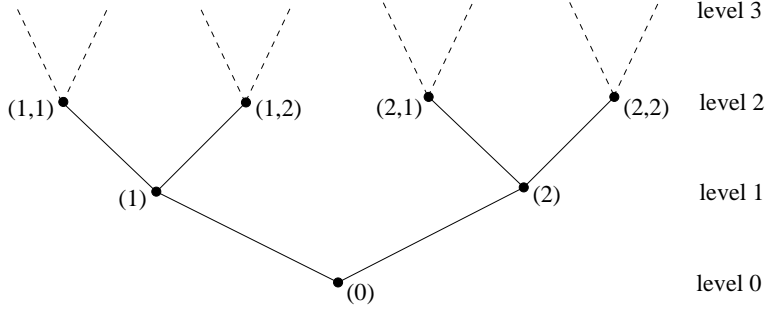}
\end{center}
\caption{The first levels of $\G_+^2$} \label{fig1}
\end{figure}

Let us define on $\G_+^k$ a binary operation $\circ:\G^k_+\times\G_+^k\to\G_+^k$
as follows, for any two elements $x=(i_1,\dots,i_n)$ and $y=(j_1,\dots,j_m)$ put
$$
x\circ y=(i_1,\dots,i_n)\circ(j_1,\dots,j_m)=(i_1,\dots,i_n,j_1,\dots,j_m)
$$
and
$$
y\circ x=(j_1,\dots,j_m)\circ(i_1,\dots,i_n)=(j_1,\dots,j_m,i_1,\dots,i_n).
$$
By means of the defined operation $\G_+^k$ becomes a
noncommutative semigroup with a unit. Using this semigroup
structure one defines translations $\tau_g:\G_+^k\to\G_+^k,\
g\in\G_k$ by
$$
\tau_g(x)=g\circ x.
$$
Similarly, by means of $\tau_g$ one can define translation $\tilde\tau_g:L\to L$ of $L$. Namely,
$$
\tilde\tau_g(\langle{x,y}\rangle)=\langle{\tau_g(x),\tau_g(y)}\rangle.
$$

Let $G\subset\G_+^k$ be a sub-semigroup of $\G_+^k$ and $h:L\to\bq_p$ be a function defined on $L$.
We say that $h$ is a $G$-{\it periodic} if $h(\tilde\tau_g(l))=h(l)$ for all $g\in G$ and $l\in L$.
Any $\G^k_+$-periodic function is called {\it translation-invariant}. Put
$$
G_m=\left\{x\in\G_+^k: d(x,x^0)\equiv0(\operatorname{mod }m)\right\},\ \ \ m\geq2.
$$

One can check that $G_m$ is a sub-semigroup with a unit.

\section{$p$-adic Ising-Vannimenus  model and its $p$-adic Gibbs measures}

In this section we consider the $p$-adic Ising-Vannimenus model
where spin takes values in the set $\Phi=\{-1,+1\}$, ($\Phi$ is
called a {\it state space}) and is assigned to the vertices of the
tree $\G^k_+=(V,\Lambda)$. A configuration $\s$ on $V$ is then
defined as a function $x\in V\to\s(x)\in\Phi$; in a similar manner
one defines configurations $\s$ and $\w$ on $V_n$ and $W_n$,
respectively. The set of all configurations on $V$ (resp. $V_n$,
$W_n$) coincides with $\Omega=\Phi^{V}$ (resp.
$\Omega_{V_n}=\Phi^{V_n},\ \ \Omega_{W_n}=\Phi^{W_n}$). One can see
that $\Om_{V_n}=\Om_{V_{n-1}}\times\Om_{W_n}$. Using this, for given
configurations $\s\in\Om_{V_{n-1}}$ and $\w\in\Om_{W_{n}}$ we
define their concatenations  by
$$
(\s\vee\w)(x)= \left\{
\begin{array}{ll}
\s(x), & \ \textrm{if} \ \  x\in V_{n-1},\\
\w(x), & \ \textrm{if} \ \ x\in W_n.\\
\end{array}
\right.
$$
It is clear that $\s\vee\w\in \Om_{V_n}$.

The Hamiltonian $H_n:\Om_{V_n}\to\bq_p$ of the $p$-adic
Ising-Vannimenus model has a form

\begin{equation}\label{ham1}
H_{n}(\sigma)=J\sum\limits_{\langle{x,y}\rangle\in
L_{n}}\sigma(x)\sigma(y)+J_{1}\sum\limits_{\rangle{x,y}\langle:
x,y\in V_{n}}\sigma(x)\sigma(y)
\end{equation}
where $J_{1},J\in B(0,p^{-1/(p-1)})$ are coupling constants.

Note that the last condition together with the strong triangle
inequality implies the existence of $\exp_p(H_n(\s))$ for all
$\s\in\Om_{V_n}$, $n\in\bn$. This is required to our construction.

Let
$\h:\langle{x,y}\rangle\to\h_{xy}=\left(h_{xy,++},h_{xy,+-},h_{xy,-+},h_{xy,--}\right)\in\bq_p^4$
be a vector valued function on $L$.

Given $n\in\bn$, let us consider a $p$-adic probability measure
$\m^{(n)}_\h$ on $\Om_{V_n}$ defined by

\begin{equation}\label{measure1}
\m_\h^{(n)}(\sigma)=\frac{1}{Z_{n}^{(\h)}}{\exp_{p}}{\left(H_{n}(\sigma)\right)}{\underset{x\in
W_{n-1},y\in
S(x)}{\prod}\left(h_{xy,\sigma(x)\sigma(y)}\right)^{\sigma(x)\sigma(y)}}
\end{equation}
Here, $\s\in\Om_{V_n}$, and $Z_n^{(\h)}$ is the corresponding
normalizing factor called a {\it partition function} given by

\begin{equation}\label{partition}
Z_{n}^{(\h)}=\sum_{\s\in\Omega_{V_n}}{\exp_{p}}{\left(H_{n}(\sigma)\right)}{\underset{x\in
W_{n-1},y\in
S(x)}{\prod}\left(h_{xy,\sigma(x)\sigma(y)}\right)^{\sigma(x)\sigma(y)}}.
\end{equation}

\begin{rem} We point out that, in general, in the definition of the
measure \eqref{measure1}, one can replace $\exp_p$ by any $p$-adic
number $\rho$. Such a kind of approach has been developed in
\cite{M12,M13}. For the sake of simplicity, we restrict ourselves
to the considered case. A more general case will be investigated
elsewhere.
\end{rem}

We recall \cite{M13,Roz} that one of the central results of the
theory of probability concerns a construction of an infinite
volume distribution with given finite-dimensional distributions,
which is a well-known {\it Kolmogorov's extension Theorem}
\cite{Sh}. Therefore, in this paper we are interested in the same
question but in a $p$-adic context. More exactly, we want to
define a $p$-adic probability measure $\m$ on $\Om$ which is
compatible with defined ones $\m_\h^{(n)}$, i.e.
\begin{equation}\label{CM}
\m(\s\in\Om: \s|_{V_n}=\s_n)=\m^{(n)}_\h(\s_n), \ \ \ \textrm{for
all} \ \ \s_n\in\Om_{V_n}, \ n\in\bn.
\end{equation}

In general, \`{a} priori the existence such a kind of measure $\m$
is not known, since there is not much information on topological
properties, such as compactness, of the set of all $p$-adic
measures defined even on compact spaces \footnote{ In the real
case, when the state space is compact, then the existence follows
from the compactness of the set of all probability measures (i.e.
Prohorov's Theorem). When the state space is non-compact, then
there is a Dobrushin's Theorem \cite{Dob1,Dob2} which gives a
sufficient condition for the existence of the Gibbs measure for a
large class of Hamiltonians.}. Note that certain properties of the
set of $p$-adic measures has been studied in \cite{kas2,kas3}, but
those properties are not enough to prove the existence of the
limiting measure. Therefore, at a moment, we can only use the
$p$-adic Kolmogorov extension Theorem (see \cite{GMR},\cite{KL})
which based on so called {\it compatibility condition} for the
measures $\m_\h^{(n)}$, $n\geq 1$, i.e.
\begin{equation}\label{comp}
\sum_{\w\in\Om_{W_n}}\m^{(n)}_\h(\s_{n-1}\vee\w)=\m^{(n-1)}_\h(\s_{n-1}),
\end{equation}
for any $\s_{n-1}\in\Om_{V_{n-1}}$. This condition according to
the theorem implies the existence of a unique $p$-adic measure
$\m$ defined on $\Om$ with a required condition \eqref{CM}. We
should stress that using the compatibility condition for the Ising
model on the Bethe lattice, in the real case, was started in
\cite{BRZ1} (see \cite{Roz} for review). Note that more general
theory of $p$-adic measures has been developed in \cite{kas1}.

Following \cite{M12,M13} if for some function $\h$ the measures
$\m_\h^{(n)}$ satisfy the compatibility condition, then there is a
unique $p$-adic probability measure, which we denote by $\m_\h$.
Such a measure $\m_\h$ is said to be {\it a $p$-adic quasi Gibbs
measure} corresponding to the $p$-adic Ising-Vannimenus model. By
$Q\cg(H)$ we denote the set of all $p$-adic quasi Gibbs measures
associated with functions $\h=\{\h_{xy},\ \langle{x,y}\rangle\in
L\}$. A $p$-adic quasi Gibbs measure defined by \eqref{measure1}
is called \textit{$p$-adic Gibbs measure} if $\h_{xy}\in\ce_{p}^4$
for all $\langle{x,y}\rangle\in L$. The set of $p$-adic Gibbs
measures we denote by $\cg(H)$. If there are at least two distinct
$p$-adic quasi Gibbs measures $\m,\n\in Q\cg(H)$ such that $\m$ is
bounded and $\n$ is unbounded, then we say that {\it a phase
transition} occurs. By another words, one can find two different
functions $\sb$ and $\h$ defined on $\bn$ such that there exist
the corresponding measures $\m_\sb$ and $\m_\h$, for which one is
bounded, another one is unbounded. Moreover, if there is a
sequence of sets $\{A_n\}$ such that $A_n\in\Om_{V_n}$ with
$|\m(A_n)|_p\to 0$ and $|\n(A_n)|_p\to\infty$ as $n\to\infty$,
then we say that there occurs a {\it strong phase transition}.

\begin{rem} We would like to point out that if one considers the usual
Ising model (i.e. in the real setting) on a multidimensional
lattice, then at low temperature there occurs a phase transition,
i.e. there exist $q$-different Gibbs measures $\m_i$,
($i=1,\dots,q$), i.e.
$$
\m_+(\s(0)=1)>1/2, \ \ \m_-(\s(0)=-1)<1/2 \ \ j\neq i.
$$
This implies that the measures $\m_\pm$ are mutually singular to
each other. The strong phase transition (see definition above), in
the $p$-adic setting, has the similar meaning as singularity, i.e.
the $p$-adic measures $\m$ and $\n$ are "singular" (in the above
given sense). Here we have to stress that absolutely continuity
and singularity of $p$-adic measures cannot be directly defined in
a similar manner with real case. Absolutely continuity of $p$-adic
measures have been studied in \cite{kas1}. The singularity what we
are proposing is consistent with that absolutely continuity
introduced in \cite{kas1}.
\end{rem}

\begin{thm}(\cite{MFDMAH})\label{compatibility}
The measures $\m^{(n)}_\h$, $ n=1,2,\dots$ (see \eqref{measure1})
satisfy the compatibility condition \eqref{comp} if and only if
for any $n\in \bn$ the following equation holds:
\begin{equation}\label{canonic3}
\begin{cases}
h_{xy,++}\cdot h_{xy,-+}=\prod\limits_{z\in S(y)}
\frac{(ab)^2h_{yz,++}h_{yz,+-}+1}{a^2h_{yz,++}h_{yz,+-}+b^2}\\
h_{xy,--}\cdot h_{xy,+-}=\prod\limits_{z\in S(y)}
\frac{(ab)^2h_{yz,--}h_{yz,-+}+1}{a^2h_{yz,--}h_{yz,-+}+b^2}\\
 h_{xy,++}\cdot h_{xy,+-}=\prod\limits_{z\in S(y)}
\frac{\big((ab)^2h_{yz,++}h_{yz,+-}+1\big)h_{yz,-+}}{\big(a^2h_{yz,--}h_{yz,-+}+b^2\big)h_{yz,+-}}\\
\end{cases}
\end{equation}
where $a=\exp_p(J)$, $b=\exp_p(J_1)$.
\end{thm}

\begin{rem} If we take $J_1=0$, i.e. $b=1$, then \eqref{canonic3}
reduces to the well-known equation for the Ising model (see
\cite{KM,MFDM} for details). For this model the existence of two
different $p$-adic quasi Gibbs measures has been considered in
\cite{Khak1}.
\end{rem}

For convenience, let us denote
\begin{equation}\label{denuh}
\begin{array}{cc}
u_{xy,1}=a^2h_{xy,++}h_{xy,-+}\\
u_{xy,2}=a^2h_{xy,--}h_{xy,-+}\\
u_{xy,3}=a^2h_{xy,++}h_{xy,+-}
\end{array}
\end{equation}
and rewrite (\ref{canonic3}) by

\begin{equation}\label{canonic_u}
\left\{\begin{array}{ll}
u_{xy,1}=a^2\prod\limits_{z\in S(y)}\frac{b^2u_{yz,3}+1}{u_{yz,3}+b^2}\\[4mm]
u_{xy,2}=a^2\prod\limits_{z\in S(y)}\frac{(b^2u_{yz,2}+1)u_{yz,3}}{(u_{yz,3}+b^2)u_{yz,1}}\\[4mm]
u_{xy,3}=a^2\prod\limits_{z\in S(y)}\frac{(b^2u_{yz,3}+1)u_{yz,1}}{(u_{yz,2}+b^2)u_{yz,3}}
 \end{array}\right.
\end{equation}

Hence, due to Theorem \ref{compatibility} the problem of
describing the $p$-adic quasi Gibbs measures is reduced to the
description of solutions of the functional equations
\eqref{canonic_u}. It is worth mentioning that there are
infinitely many solutions of the system of equations
\eqref{canonic3} corresponding to each solution of the system of
equations \eqref{canonic_u}. However, each solution of the system
of equations \eqref{canonic_u} uniquely determines a $p$-adic
quasi Gibbs measure.

\begin{thm}\label{pro1}
There exists a unique $p$-adic  quasi Gibbs measure $\mu_\mathbf{u}$ associated
with the function $\mathbf{u}=\{\mathbf{u}_{xy}, \
\langle{x,y}\rangle\in L \}$  where $\mathbf{u}_{xy}=(u_{xy,1},u_{xy,2},u_{xy,3})$
is a solution of the system of equations \eqref{canonic_u}. Moreover, the measure
$\mu_\mathbf{u}$ is the $p$-adic Gibbs measure if and only if
$\mathbf{u}_{xy}\in\ce^3_p$ for all $\langle x,y\rangle\in L$.
\end{thm}

\begin{proof}
Let $\mathbf{u}=\{\mathbf{u}_{xy}, \ \langle{x,y}\rangle\in L \}$
be a  function, where
$\mathbf{u}_{xy}=(u_{xy,1},u_{xy,2},u_{xy,3})$ is a solution of
the system of equations \eqref{canonic_u}. Then, for any $p$-adic
number $h_{xy,++}\in\bq_p\setminus\{0\}$, a function
$\mathbf{h}=\{\mathbf{h}_{xy},\ \langle{x,y}\rangle\in L\}$
defined by
$$
\mathbf{h}_{xy}=\left(h_{xy,++},\frac{u_{xy,3}}{a^2h_{xy,++}},
\frac{u_{xy,1}}{a^2h_{xy,++}},\frac{u_{xy,2}h_{xy,++}}{u_{xy,1}}\right)
$$
is a solution of \eqref{canonic3}.

Now fix $n\geq1$. Since $|W_{n-1}|=k^{n-1}$ and $|S(x)|=k$ we get
$|L_{n}\setminus L_{n-1}|=k^n$. Let $\s$ be any configuration in
$\Om_{V_n}$. Denote
\[
\begin{array}{ll}
\cn_{1,n}(\s)=\{\langle{x,y}\rangle\in{L_n\setminus{L_{n-1}}}:\ \s(x)=1,\ \s(y)=1,\ x\in W_{n-1},\ y\in S(x)\}\\
\cn_{2,n}(\s)=\{\langle{x,y}\rangle\in{L_n\setminus{L_{n-1}}}:\ \s(x)=1,\ \s(y)=-1,\ x\in W_{n-1},\ y\in S(x)\}\\
\cn_{3,n}(\s)=\{\langle{x,y}\rangle\in{L_n\setminus{L_{n-1}}}:\ \s(x)=-1,\ \s(y)=1,\ x\in W_{n-1},\ y\in S(x)\}\\
\cn_{4,n}(\s)=\{\langle{x,y}\rangle\in{L_n\setminus{L_{n-1}}}:\ \s(x)=-1,\ \s(y)=-1,\ x\in W_{n-1},\ y\in S(x)\}
\end{array}
\]

We have
\begin{eqnarray*}
\prod_{x\in W_{n-1}\atop{y\in
S(x)}}(h_{xy,\s(x)\s(y)})^{\s(x)\s(y)}&=&\prod\limits_{\langle
x,y\rangle\in\cn_{1,n}(\s)}h_{xy,++} \prod\limits_{\langle
x,y\rangle\in\cn_{2,n}(\s)}\frac{a^2h_{xy,++}}{u_{xy,3}}\prod\limits_{\langle
x,y\rangle\in\cn_{3,n}(\s)}\frac{a^2h_{xy,++}}{u_{xy,1}}\\[2mm]
&&\times \prod\limits_{\langle
x,y\rangle\in\cn_{4,n}(\s)}\frac{u_{xy,2}h_{xy,++}}{u_{xy,1}}\\[2mm]
& =&\prod\limits_{\langle x,y\rangle\in L_n\setminus
L_{n-1}}h_{xy,++}\prod\limits_{\langle
x,y\rangle\in\cn_{2,n}(\s)}\frac{a^2}{u_{xy,3}}
\prod\limits_{\langle
x,y\rangle\in\cn_{3,n}(\s)}\frac{a^2}{u_{xy,1}}\\[2mm]
&&\times \prod\limits_{\langle
x,y\rangle\in\cn_{4,n}(\s)}\frac{u_{xy,2}}{u_{xy,1}}
\end{eqnarray*}
By means of the last equalities, one can get from (\ref{measure1})
and (\ref{partition}) that
\begin{eqnarray}\label{mu_u}
\mu_{h}^{(n)}(\s)&=&\frac{\exp_p(H_n(\s))\prod\limits_{x\in
W_{n-1}\atop{y\in S(x)}}(h_{xy,\s(x)\s(y)})^{\s(x)\s(y)}}
{\sum\limits_{\w\in\Om_{V_n}}\exp_p(H_n(\w))\prod\limits_{x\in
W_{n-1}\atop{y\in S(x)}}(h_{xy,\w(x)\w(y)})^{\w(x)\w(y)}}
\nonumber\\[2mm]
& =&\frac{\exp_p(H_n(\s))\prod\limits_{\langle
x,y\rangle\in\cn_{2,n}(\s)}\frac{a^2}{u_{xy,3}}
\prod\limits_{\langle
x,y\rangle\in\cn_{3,n}(\s)}\frac{a^2}{u_{xy,1}}
\prod\limits_{\langle
x,y\rangle\in\cn_{4,n}(\s)}\frac{u_{xy,2}}{u_{xy,1}}}
{\sum\limits_{\w\in\Om_{V_n}}\exp_p(H_n(\w))\prod\limits_{\langle
x,y\rangle\in\cn_{2,n}(\w)}\frac{a^2}{u_{xy,3}}
\prod\limits_{\langle
x,y\rangle\in\cn_{3,n}(\w)}\frac{a^2}{u_{xy,1}}
\prod\limits_{\langle
x,y\rangle\in\cn_{4,n}(\w)}\frac{u_{xy,2}}{u_{xy,1}}}
\end{eqnarray}

One can see the right hand side of \eqref{mu_u} does not depend on
$h_{xy,++}$. So, we can see that each solution $\mathbf{u}$ of the
system of equations \eqref{canonic_u} uniquely determines only one
$p$-adic quasi Gibbs measure $\mu_\mathbf{u}$.
\end{proof}

\section{The existence of $p$-adic quasi Gibbs Measures}

In this section we are going to establish the existence of
$p$-adic quasi Gibbs measures by analyzing the equation
\eqref{canonic_u}.

Recall that a $\G_+^k$-periodic function is called
\textit{translation-invariant}. Let ${\bf u}=\{{\bf
u}_{xy}\}_{\langle{x,y}\rangle\in L}$ be a translation-invariant
function i.e. $\mathbf{u}_{xy}=\mathbf{u}_{zw}$ for all
$\langle{x,y}\rangle,\langle{z,w}\rangle\in L$. A $p$-adic quasi
Gibbs measure $\m_{\bf u}$, corresponding to a
translation-invariant function ${\bf u}$, is called a {\it
translation-invariant $p$-adic quasi Gibbs measure}. The set of
translation-invariant $p$-adic quasi Gibbs (resp. $p$-adic Gibbs)
measures is denoted by $Q\cg(H)_{TI}$ (resp. $\cg(H)_{TI}$).

To solve the equation \eqref{canonic_u}, in general, is very
complicated. Therefore,
 let us first restrict ourselves to
the description of translation-invariant solutions of
\eqref{canonic_u}. Therefore, \eqref{canonic_u} reduces to
\begin{equation}\label{tru}
\left\{\begin{array}{ll}
u_1=a^2\left(\frac{b^2u_3+1}{u_3+b^2}\right)^k\\[3mm]
u_2=a^2\left(\frac{(b^2u_2+1)u_3}{(u_3+b^2)u_1}\right)^k\\[3mm]
u_3=a^2\left(\frac{(b^2u_3+1)u_1}{(u_2+b^2)u_3}\right)^k
\end{array}\right.
\end{equation}
In what follows, we will always assume that $p\geq 3$ and $b\neq1$.

\subsection{The study of the system \eqref{tru}}

In this subsection we are aiming to study the set of all solutions
of the system \eqref{tru}.

Let $\mathbf{u}=(u_1,u_2,u_3)$ and
$F(\mathbf{u})=(F_1(\mathbf{u}),F_2(\mathbf{u}),F_3(\mathbf{u}))$
be a mapping, where
$$F_1(\mathbf{u})=a^2\left(\frac{b^2u_3+1}{u_3+b^2}\right)^k, \quad F_2(\mathbf{u})=
a^2\left(\frac{(b^2u_2+1)u_3}{(u_3+b^2)u_1}\right)^k, \quad F_3(\mathbf{u})=
a^2\left(\frac{(b^2u_3+1)u_1}{(u_2+b^2)u_3}\right)^k.$$

Clearly that the solutions of \eqref{tru} are fixed points of the
mapping $F$. Therefore, the  set $\textbf{Fix}(F)=\{\mathbf{u}:
F(\mathbf{u})=\mathbf{u}\}$ describes all solutions of the system
\eqref{tru}.

\begin{prop}\label{solutionnorm}
If $\mathbf{u}=(u_1,u_2,u_3)\in\textbf{Fix}(F)$ then one has that $|u_2|_p=|u_3|_p=1$.
\end{prop}
\begin{proof} Let $\mathbf{u}=(u_1,u_2,u_3)\in\textbf{Fix}(F)$ be a solution of \eqref{tru}.

Let us first prove that $|u_3|_p=1$.  Assume that $|u_3|_p\neq 1$.
In this case, due to $|b^2u_3+1|_p=|u_3+b^2|,$ from the first
equation of \eqref{tru} we obtain that $|u_1|_p=1$. By
substituting it, into the second and the third equations of
\eqref{tru}, one finds
\begin{equation}\label{1norm}
\left\{\begin{array}{ll}
|u_2|_p=|b^2u_2+1|_p^k\cdot|u_3|_p^k\\[3mm]
|u_3|_p=|u_2+b^2|_p^{-k}\cdot|u_3|_p^{-k}
\end{array}\right.\ \ \ \ \mbox{whenever}\ \ |u_3|_p<1
\end{equation}
and
\begin{equation}\label{2norm}
\left\{\begin{array}{ll}
|u_2|_p=|b^2u_2+1|_p^k\\[3mm]
|u_3|_p=|u_2+b^2|_p^{-k}
\end{array}\right.\ \ \ \ \mbox{whenever}\ \ |u_3|_p>1
\end{equation}

It follows from the first equality of \eqref{1norm} that
$|u_2|_p\neq1$. Then, by multiplying both equalities of
\eqref{1norm}, one gets that
 $|u_2u_3|_p=1$. Therefore, one has that $|u_2|_p=|u_3|_p^{-1}>1$. However,
 by substituting the last one into the second equality of \eqref{1norm}, we
 have that $|u_3|_p=1$ which is a contradiction.

Similarly, from the first equality of \eqref{2norm}, we
immediately get that $|u_2|_p=|b^2u_2+1|_p=1$. It then follows
that $
|u_2+b^2|_p=|b^2u_2+b^4|_p=|(b^2u_2+1)+(b^4-1)|_p=|b^2u_2+1|_p=1.
$ However, by substituting the last equality into the second
equation of \eqref{2norm}, one gets that $|u_3|_p=1$ which is
again a contradiction. These contradictions imply that
$|u_3|_p=1$.

Now, we are going to show that $|u_2|_p=1$.  Again assume the
contrary, i.e.  $|u_2|_p\neq1$. In this case, we have
$|b^2u_2+1|_p=|u_2+b^2|$. By multiplying the second and the third
equations of \eqref{tru}, one finds
$\frac{u_2u_3}{u_1}=a^2\left(\frac{b^2u_2+1}{u_2+b^2}\right)^k$.
The last equality with $|u_3|_p=1$ yields that
$|u_1|_p=|u_2|_p\neq 1$. It then follows from the first equation
of \eqref{tru} that
$|u_1|_p=\frac{|b^2u_3+1|_p}{|u_3+b^2|_p}\neq1$. On the other
hand, one has that $|b^2u_3+1|_p=1$ if and only if
$|u_3+b^2|_p=1$. Therefore, it implies that $|b^2u_3+1|_p<1$. From
this inequality together with the third equation of \eqref{tru}
one finds that $
|u_3|_p<\left(\frac{|u_1|_p}{|u_2+b^2|_p}\right)^k. $ Since
$|u_1|_p=|u_2|_p\neq 1,$ one has that $
\left(\frac{|u_1|_p}{|u_2+b^2|_p}\right)^k\leq1. $ Hence, we
obtain that $|u_3|_p<1$ which is a contradiction. This completes
the proof.
\end{proof}

\begin{prop}\label{solutionnorm1}
The following statements hold:\\
$(i)$ One has that $\textbf{Fix}(F)\subset(\ce_p\times\ce_p\times\ce_p)\bigcup(\bq_p\times(-\ce_p)\times(-\ce_p))$;\\
$(ii)$ If $\textbf{Fix}(F)\setminus(\ce_p\times\ce_p\times\ce_p)\neq\emptyset$ then $-1$
is the $k^{th}$ residue of module $p$.
\end{prop}

\begin{proof}
$(i)$ Let $\mathbf{u}=(u_1,u_2,u_3)\in\textbf{Fix}(F)$ be a
solution of the system \eqref{tru}.

We first show that if $u_3\not\in-\ce_p$ then $\mathbf{u}\in\ce_p^3$. Indeed, due to
Proposition \ref{solutionnorm}, we have that $|u_2|_p=|u_3|_p=1,$ $|u_3+1|_p=1$, and
$$
\frac{b^2u_3+1}{u_3+b^2}=\frac{b^2+\frac{1-b^2}{u_3+1}}{1+\frac{b^2-1}{u_3+1}}\in\ce_p.
$$
By using these, from the first equation of \eqref{tru} one finds
that $u_1\in\ce_p$. Due to $|u_1|_p=|u_2|_p=|u_3|_p=1$ together
with the second equation of \eqref{tru} we get that
$|b^2u_2+1|_p=1$, i.e. $u_2\not\in-\ce_p$. From
$$
u_2u_3=a^2u_1\left(\frac{b^2u_2+1}{u_2+b^2}\right)^k, \quad \frac{b^2u_2+1}{u_2+b^2}=
\frac{b^2+\frac{1-b^2}{u_2+1}}{1+\frac{b^2-1}{u_2+1}}\in\ce_p
$$
it follows that $u_2\cdot u_3\in\ce_p$. This means that
$$
\frac{b^2u_3+1}{(u_2+b^2)u_3}=\frac{b^2u_3+1}{u_2u_3+b^2u_3}=\frac{1}{1+\frac{u_2u_3-1}{b^2u_3+1}}\in\ce_p.
$$
Hence, from the third equation of \eqref{tru} one gets
$u_3\in\ce_p$, and consequently, $u_2\in\ce_p$.

Now, we want to show that if $u_3\in-\ce_p$ then $u_2\in-\ce_p$.
We assume the contrary, i.e. $u_2\not\in-\ce_p$. From
$|u_2|_p=|u_3|_p=1$ and
$$
u_2u_3=a^2u_1\left(\frac{b^2u_2+1}{u_2+b^2}\right)^k, \quad \frac{b^2u_2+1}{u_2+b^2}=
\frac{b^2+\frac{1-b^2}{u_2+1}}{1+\frac{b^2-1}{u_2+1}}\in\ce_p
$$
we get that $|u_1|_p=1$. By substituting it, into the third
equation of \eqref{tru}, one has that $|b^2u_3+1|_p=1$ which
contradicts to $u_3\in-\ce_p$. This means that $u_2\in-\ce_p$.

$(ii)$ Assume that a solution $\mathbf{u}=(u_1,u_2,u_3)$ of the
system \eqref{tru} does not belong to $\ce_p^3$. Then we want to
show that $-1$ is the $k^{th}$ residue of module $p$. Again we
suppose the contrary, i.e. $\alpha^k\not\equiv-1(\operatorname{mod
}p)$ for all $\alpha\in\{1,2,\dots,p-1\}$. Due to Proposition
\ref{solutionnorm}, we have that $|u_2|_p=|u_3|_p=1$. Let
$y=\frac{u_3}{a^2}$ and $z=\frac{(b^2u_3+1)u_1}{(u_2+b^2)u_3}$. It
is clear that $|y|_p=1$ and it follows form the third equation of
\eqref{tru} that $z^k=y$, so, $|z|_p=1$. Therefore, one has
$z=z_0+z_1p+z_2p^2+\cdots$ and $y=y_0+y_1p+y_2p^2+\cdots$ such
that $y_0\equiv z_0^k\not\equiv-1(\operatorname{mod }p)$.
Consequently, from $u_3=a^2y$ we obtain that $u_3\notin-\ce_p$.
This, according to the previous case, yields that
$\mathbf{u}\in\ce_p^3$ which is a contradiction. This completes
the proof.
\end{proof}

\begin{rem} According to Theorem \ref{pro1}, if $-1$ \textit{is not} the $k^{th}$ residue of
module $p$ then any translation-invariant $p$-adic quasi Gibbs measure is a $p$-adic Gibbs measure.
\end{rem}

Our next aim is to show that
$\textbf{Fix}(F)\cap(\ce_p\times\ce_p\times\ce_p)\neq\emptyset.$
For that purpose, we are searching for a solution of the system of
equations \eqref{tru} in the form $u_1=u_2=u_3=u$. In that case,
the system  \eqref{tru} reduces to
\begin{equation}\label{u_1=u_2=_u3}
u=a^2\left(\frac{b^2u+1}{u+b^2}\right)^k.
\end{equation}

Let us consider a function $f_{a,b,k}:\bq_p\to\bq_p$ defined by
\begin{eqnarray}\label{function}
f_{a,b,k}(u)=a^2\left(\cfrac{b^2u+1}{u+b^2}\right)^k, \quad b\neq1
\end{eqnarray}

Note that in the real setting, the fixed points and dynamics of
the function $f$ defined by \eqref{function} were studied very
well in the classical textbooks of statistical mechanics (for the
latest book see \cite{Roz}). Such a function \eqref{function} is
called \textit{the Ising-Potts mapping}. The Ising-Potts mapping
may exhibit a chaotic behavior (see \cite{BGR}, \cite{Monroe}).
Here, in the $p$-adic field, we would like to examine the set
$\textbf{Fix}(f_{a,b,k})=\{u\in\bq_p: f_{a,b,k}(u)=u \}$ of fixed
points of the function \eqref{function}. The dynamics of the
function $f_{a,b,k}$ over the $p$-adic field will be studied
elsewhere.

\begin{thm}\label{fixf}
Let $p\geq3$, $a,b\in\ce_p,\ b\neq1$ and
$f_{a,b,k}:\bq_p\to\bq_p$ be a function defined by
\eqref{function}. Then the following statements hold:

\begin{itemize}
\item[$(i)$] One has that
$\textbf{Fix}(f_{a,b,k})\subset\ce_p\cup(-\ce_p)$ and
$|\textbf{Fix}(f_{a,b,k})\cap\ce_p|=1;$ \item[$(ii)$]  If
$\frac{p-1}{(k,p-1)}$ is odd then
$\textbf{Fix}(f_{a,b,k})\cap(-\ce_p)=\emptyset;$ \item[$(iii)$] If
$|k|_p=1$ and $\frac{p-1}{(k,p-1)}$ is even then
$|\textbf{Fix}(f_{a,b,k})\cap(-\ce_p)|= (k,p-1).$ \item[$(iv)$] If
$k$ is odd and $\max\{|k|_p, |a-1|_p\}<|b-1|_p$ then
$|\textbf{Fix}(f_{a,b,k})\cap(-\ce_p)|\geq 1.$
\end{itemize}
\end{thm}

\begin{proof} $(i)$ Due to Proposition \ref{solutionnorm}, one has that
$\textbf{Fix}(f_{a,b,k})\subset\ce_p\cup(-\ce_p)$. Let us show
that $\left|\mbox{\textbf{Fix}}(f_{a,b,k})\cap\ce_p\right|=1.$
Putting $y=\frac{u}{a^2}$ from \eqref{u_1=u_2=_u3} one gets that
\begin{equation}\label{y=x^k}
y=\left(\frac{a^2b^2y+1}{a^2y+b^2}\right)^k.
\end{equation}
This yields that any solution $y$ of the equation \eqref{y=x^k} is
the $k^{th}$ power of some $p$-adic number $x\in\bq_p$, i.e.
$y=x^k$. It follows from \eqref{y=x^k} that
\begin{equation}\label{x^k}
x^k=\left(\frac{a^2b^2x^k+1}{a^2x^k+b^2}\right)^k\quad \mbox{or} \quad \left(\frac{a^2b^2x^k+1}{a^2x^{k+1}+xb^2}\right)^k=1
\end{equation}
Let $\eta=\frac{a^2b^2x^k+1}{a^2x^{k+1}+xb^2}$, then $\eta$ is a
solution of the equation $\eta^k=1$. The last equation has the
following solutions $\eta_1 (=1), \eta_2, \cdots, \eta_d$ in
$\bq_p$, where $d=(k,p-1).$ Therefore, we have
\begin{equation}\label{etaneq1}
x=\eta_i\frac{a^2b^2x^k+1}{a^2x^k+b^2}, \quad i=\overline{1,d}
\end{equation}
As for two different $i_1\neq i_2$, the corresponding solutions $x_{i_1}, x_{i_2}$
of the equation \eqref{etaneq1} satisfy the condition $\frac{x_{i_1}}{\eta_{i_1}}=
\frac{x_{i_2}}{\eta_{i_2}}$. Therefore, it is enough to study the equation \eqref{etaneq1} for $\eta_1=1$.

Let $g_{a,b,k}:\bq_p\to\bq_p$ be a function given by
\begin{equation}\label{g(x)}
g_{a,b,k}(x)=\frac{a^2b^2x^k+1}{a^2x^k+b^2}.
\end{equation}
Consequently, we conclude that
\begin{equation}\label{fixf=fixg}
\mbox{\textbf{Fix}}(f_{a,b,k})=\left\{u: u=a^2x^k,\
x\in\mbox{Fix}(g_{a,b,k})\right\}.
\end{equation}

Due to $\textbf{Fix}(f_{a,b,k})\subset\ce_p\cup(-\ce_p)$, it
follows from \eqref{fixf=fixg} that
$$
\textbf{Fix}(g_{a,b,k})\subset\left\{x\in\bz^{*}_p: x_0^k\equiv
1(\operatorname{mod }p)\right\} \cup\left\{x\in\bz^{*}_p:
x_0^k\equiv-1(\operatorname{mod }p)\right\}.
$$
Let $x\in \textbf{Fix}(g_{a,b,k})\cap\bz^{*}_p$ such that
$x_0^k\equiv1(\operatorname{mod }p)$. We then get that
$\left|x^k+1\right|_p=1$. This implies that
\begin{equation}\label{gincep}
x=\frac{a^2b^2x^k+1}{a^2x^k+b^2}=b^2\frac{1-\frac{a^2b^2-1}{a^2b^2(x^k+1)}}{1+\frac{b^2-a^2}{a^2(x^k+1)}}\in\ce_p.
\end{equation}
Thus, we obtain
\begin{equation}\label{fixg}
\textbf{Fix}(g_{a,b,k})\subset\ce_p\cup\left\{x\in\bz^{*}_p:
x_0^k\equiv-1(\operatorname{mod }p)\right\}.
\end{equation}

It is clear that $g_{a,b,k}(\ce_p)\subset\ce_p$. For any
$x,y\in\ce_p$, we get that
\begin{equation}\label{f(x)-f(y)}
g_{a,b,k}(x)-g_{a,b,k}(y)=\frac{a^2(b^4-1)\sum\limits_{i=0}^{k-1}x^{k-1-i}y^i}{(a^2x^k+b^2)(a^2y^k+b^2)}(x-y).
\end{equation}
According to Lemma \ref{epproperty} we have that
\begin{equation}\label{normss}
\left\{\begin{array}{ll}
\left|b^4-1\right|_p\leq\frac{1}{p},& \left|\sum\limits_{i=0}^{k-1}x^{k-1-i}y^i\right|_p\leq1,\\
\left|a^2x^k+b^2\right|_p=1,& \left|a^2y^k+b^2\right|_p=1.
\end{array}\right.
\end{equation}
By using \eqref{normss} and \eqref{f(x)-f(y)}, we obtain that
$$
\left|g_{a,b,k}(x)-g_{a,b,k}(y)\right|_p\leq\frac{1}{p}|x-y|_p.
$$
This means that $g_{a,b,k}$ is a contraction on $\ce_p$. Since
$\ce_p$ is compact, the function $g_{a,b,k}$ has a unique fixed
point in $\ce_p$. Hence,  due to \eqref{fixf=fixg}, we have that
$\left|\textbf{Fix}(f_{a,b,k})\cap\ce_p\right|=1.$

$(ii)$ We know that $\sqrt[k]{-1}$ does not exist in $\mathbb F_p$
if and only if $\frac{p-1}{(k,p-1)}$ is odd. In this case, due to
\eqref{fixf=fixg} and \eqref{fixg}, we have that
$\textbf{Fix}(f_{a,b,k})\cap(-\ce_p)=\emptyset;$.

$(iii)$ Let $(k,p)=1$ and $\frac{p-1}{(k,p-1)}$ is even.  Let us
establish that
$\left|\textbf{Fix}(f_{a,b,k})\cap(-\ce_p)\right|=(k,p-1).$ Under
our assumption, $\sqrt[k]{-1}$ exists in $\mathbb F_p$. Let
$\alpha$ be some solution of the congruent equation
$\alpha^k+1\equiv 0 (mod \ p)$.   From $g_{a,b,k}(x)=x$ one finds
that
\begin{equation}\label{g(x)=x}
\frak{g}(x)\equiv a^2x^{k+1}-a^2b^2x^k+b^2x-1=0.
\end{equation}
Since $\alpha\not\equiv1(\operatorname{mod }p)$ and $(k,p)=1$,
we then have that
\begin{eqnarray*}
\frak{g}(\alpha)=-a^2\alpha+a^2b^2+b^2\alpha-1=\alpha(b^2-a^2)+a^2b^2-1\equiv 0(\operatorname{mod }p),\\
\alpha
\frak{g}'(\alpha)=-(k+1)a^2\alpha+ka^2b^2+b^2\alpha=\alpha(b^2-a^2)+ka^2(b^2-\alpha)\not\equiv
0(\operatorname{mod }p).
\end{eqnarray*}
According to the Hensel's lemma there exists a unique $p$-adic integer $x_\alpha$ such that
$$\frak{g}(x_\alpha)=0,\ \ \ \ x_\alpha\equiv\alpha(\operatorname{mod }p).$$
Consequently, for each solution of the congruent equation
$\alpha^k+1\equiv0(\operatorname{mod }p)$, there exists a unique
solution $x_{\alpha}$ of $g_{a,b,k}(x)=x$ such that
$x_\alpha\equiv\alpha(\operatorname{mod }p)$. Hence
$\left|\textbf{Fix}(g_{a,b,k})\cap(-\ce_p)\right|=(k,p-1)$ or
equivalently $
\left|\textbf{Fix}(f_{a,b,k})\cap(-\ce_p)\right|=(k,p-1). $

$(iv)$ Let $k$ be odd and $r=|b-1|_p$. We then have that
$f_{a,b,k}(-1)=-a^2$ and
\begin{equation}\label{fu-fv}
f_{a,b,k}(u)-f_{a,b,k}(v)=\frac{a^2(b^4-1)\sum\limits_{i=0}^{k-1}[(b^2u+1)(v+b^2)]^{k-1-i}[(b^2v+1)(u+b^2)]^{i}}{(u+b^2)^k(v+b^2)^k}(u-v).
\end{equation}
for all $u,v\in B(-1,r)$. From $|u+1|_p<|b-1|_p$ and
$|v+1|_p<|b-1|_p$ one has
\begin{equation}\label{normsuv}
\left\{\begin{array}{ll}
\left|b^2u+1\right|_p=\left|b^2v+1\right|_p=\left|u+b^2\right|_p=\left|v+b^2\right|_p=|b-1|_p,\\[3mm]
\left|\sum\limits_{i=0}^{k-1}[(b^2u+1)(v+b^2)]^{k-1-i}[(b^2v+1)(u+b^2)]^{i}\right|_p\leq\left|k(b-1)^{2k-2}\right|_p.
\end{array}\right.
\end{equation}
By plugging \eqref{normsuv} into \eqref{fu-fv}, we obtain
\begin{equation}\label{derf}
\left|f_{a,b,k}(u)-f_{a,b,k}(v)\right|_p\leq\frac{|k|_p}{|b-1|_p}\cdot|u-v|_p.
\end{equation}
Since $\max\{|k|_p,|a-1|_p\}<|b-1|_p$, one has that $-a^2\in
B(-1,r)$ and $f_{a,b,k}(B(-1,r))\subset B(-a^2,r_1)\subset
B(-1,r)$ with $r_1<r$. Thus, the function $f_{a,b,k}$ is a
contraction on $B(-1,r)$. Consequently, there exists a unique
fixed point in $B(-1,r)\subset(-\ce_p)$, i.e.
$\left|\textbf{Fix}(f_{a,b,k})\cap(-\ce_p)\right|\geq 1.$
\end{proof}

\begin{rem}
We stress that the equality
$\left|\textbf{Fix}(f_{a,b,k})\cap\ce_p\right|=1$ can be directly
proven by means of the results of \cite{MA3}.
\end{rem}

The following result immediately follows from Theorem \ref{fixf}.

\begin{cor}\label{cor(k,p)} Let $p\geq3$ and $a,b\in\ce_p,\ b\neq1$. The following statements hold true:

\begin{itemize}
\item[$(i)$] One has that $|\textbf{Fix}(F)\cap(\ce_p\times\ce_p\times\ce_p)|\geq1;$
\item[$(ii)$]  If $\frac{p-1}{(k,p-1)}$ is odd then  $\textbf{Fix}(F)\cap(\bq_p\times(-\ce_p)\times(-\ce_p))=\emptyset;$
\item[$(iii)$] If $|k|_p=1$ and $\frac{p-1}{(k,p-1)}$ is even then $|\textbf{Fix}(F)\cap((-\ce_p)\times(-\ce_p)\times(-\ce_p))|\geq (k,p-1).$
\item[$(iv)$] If $k$ is odd and $\max\{|k|_p, |a-1|_p\}<|b-1|_p$ then  $|\textbf{Fix}(F)\cap((-\ce_p)\times(-\ce_p)\times(-\ce_p))|\geq 1.$
\end{itemize}
\end{cor}

The following theorem gives the precise description of the set
$\textbf{Fix}(F)\cap(\ce_p\times\ce_p\times\ce_p)$.

\begin{thm} \label{unique(k,p)}
Let $p\geq3$ and $a,b\in\ce_p,\ b\neq1$. One has that
$|\textbf{Fix}(F)\cap(\ce_p\times\ce_p\times\ce_p)|=1$ and
$$\textbf{Fix}(F)\cap(\ce_p\times\ce_p\times\ce_p)=(\textbf{Fix}(f_{a,b,k})\cap\ce_p)\times(\textbf{Fix}(f_{a,b,k})\cap\ce_p)\times(\textbf{Fix}(f_{a,b,k})\cap\ce_p).$$
\end{thm}

\begin{proof}
In order to describe the set
$\textbf{Fix}(F)\cap(\ce_p\times\ce_p\times\ce_p)$, we have to
solve the system of equation \eqref{tru} in the set
$\ce_p\times\ce_p\times\ce_p$.  Let us define the polynomial
function $G:\bq^3_p\to \bq^3_p$
$G(\textbf{u})=(G_1(\textbf{u}),G_2(\textbf{u}),G_3(\textbf{u}))$
with $p-$adic integer coefficients as follows
$$
\begin{array}{ll}
G_1(\textbf{u})=u_1(u_3+b^2)^k-a^2(b^2u_3+1)^k,\\
G_2(\textbf{u})=u_1^ku_2(u_3+b^2)^k-a^2(b^2u_2+1)^ku_3^k,\\
G_3(\textbf{u})=(u_2+b^2)^ku_3^{k+1}-a^2u_1^k(b^2u_3+1)^k.
\end{array}
$$
It is clear that
$\textbf{N}_{G}(\ce_p\times\ce_p\times\ce_p)=\textbf{Fix}(F)\cap(\ce_p\times\ce_p\times\ce_p)$
where
$\textbf{N}_{G}(\ce_p\times\ce_p\times\ce_p)=\{\mathbf{u}\in\ce_p\times\ce_p\times\ce_p:
G(\mathbf{u})=\Theta\}$. So, we apply the generalized Hensel's
lemma in order to describe the set
$\textbf{N}_{G}(\ce_p\times\ce_p\times\ce_p)$.

It is easy to check that $G(1,1,1)\equiv \Theta \ ({mod} \ p)$ and
$det(\mathbb{J}(G(1,1,1))\equiv 2^{3k}\not\equiv0\ (mod \ p)$,
where
$$ \mathbb{J}(G(1,1,1)\equiv \left(
\begin{array}{ccc}
2^k & 0 & 0\\
2^kk & 2^{k-1}(2-k) & -2^{k-1}k\\
-2^kk & 2^{k-1}k & 2^{k-1}(k+2)
\end{array}\right) \ (mod \ p)$$

Thus, due to the generalized Hensel's Lemma \ref{hl} there exist a
unique root $\mathbf{u}_*$ of the polynomial equation
$G(\mathbf{u})=\Theta$ which satisfies the condition
$\textbf{u}_*\equiv(1,1,1) \ (mod \ p)$. Consequently,
$|\textbf{N}_{G}(\ce_p\times\ce_p\times\ce_p)|=|\textbf{Fix}(F)\cap(\ce_p\times\ce_p\times\ce_p)|=1.$

On the other hand, we already knew that the system of equation
\eqref{tru} has a solution of the form $\mathbf{u}=(u,u,u)$ in the
set $\ce_p\times\ce_p\times\ce_p$, where $u$ is a unique fixed of
the function $f$ given by \eqref{function} in the set $\ce_p$.
Therefore, we get that
$\textbf{Fix}(F)\cap(\ce_p\times\ce_p\times\ce_p)=(\textbf{Fix}(f_{a,b,k})\cap\ce_p)\times(\textbf{Fix}(f_{a,b,k})\cap\ce_p)\times(\textbf{Fix}(f_{a,b,k})\cap\ce_p).$
\end{proof}

\subsection{The existence of translation-invariant $p$-adic quasi Gibbs measures}

By means of Theorem \ref{pro1} and Corollary \ref{cor(k,p)}, we
may construct translation invariant $p$-adic quasi Gibbs measures.
\begin{thm}\label{existenceGM}
Let $p\geq3$ and $|J|_p\leq\frac{1}{p},\ 0<|J_1|_p\leq\frac{1}{p}$. The following
assertions hold true for the $p$-adic Ising-Vannimenus model \eqref{ham1} on a Cayley tree of any order $k$:
\begin{itemize}
\item[$(i)$] There always exists at least one
translation-invariant $p$-adic Gibbs measure, i.e.
$|\cg(H)_{TI}|\geq 1$; \item[$(ii)$] If $\frac{p-1}{(k,p-1)}$ is
odd then there is no translation-invariant $p$-adic quasi Gibbs
measure, i.e. $Q\cg(H)_{TI}\setminus \cg(H)_{TI}=\emptyset$; \item
[$(iii)$] If $|k|_p=1$ and $\frac{p-1}{(k,p-1)}$ is even then
there are at least $(k,p-1)-$ number of translation-invariant
$p$-adic quasi Gibbs measures, i.e. $|Q\cg(H)_{TI}\setminus
\cg(H)_{TI}|\geq (k,p-1)$; \item[$(iv)$] If $k$ is odd and
$\max\{|k|_p, |a-1|_p\}<|b-1|_p$ then there is at least one
translation-invariant $p$-adic quasi Gibbs measure, i.e.
$|Q\cg(H)_{TI}\setminus \cg(H)_{TI}|\geq 1$
\end{itemize}
\end{thm}

\subsection{The existence of periodic $p$-adic quasi Gibbs measures}

All previous results were concerned with translation-invariant
$p$-adic (quasi) Gibbs measures. In this section, we are aiming to
show an existence of periodic (non-translation invariant) $p$-adic
Gibbs measures for the $p$-adic Ising-Vannimenus model on a Cayley
tree of order two. Recall that $ G_2=\left\{x\in\G_+^2:
d(x,x^0)\equiv0(\operatorname{mod }2)\right\} $ is a sub-semigroup
of $\G_+^2$. The function ${\bf u}=\{{\bf
u}_{xy}\}_{\langle{x,y}\rangle\in L}$ is called two periodic if
$\textbf{u}(\tilde\tau_g(l))=\textbf{u}(l), \forall \ g\in G_2,
l\in L$.

\begin{thm}\label{perGM}
Let $p\equiv1\ (mod \ 4)$ and $|J|_{p}\leq\frac{1}{p}$,
$0<\left|J_{1}\right|_{p}\leq\frac{1}{p}$. Then there exist
2-periodic $p$-adic quasi Gibbs measures for the $p$-adic
Ising-Vannimenus model on a Cayley tree of order two.
\end{thm}

\begin{proof}
In order to describe the 2-periodic $p$-adic quasi Gibbs measures,
we have to solve the following system of equations
\begin{equation}\label{per}
\left\{\begin{array}{ll}
u_1=a^2\left(\frac{b^2v_3+1}{v_3+b^2}\right)^2,& \ \ \ v_1=a^2\left(\frac{b^2u_3+1}{u_3+b^2}\right)^2,\\[3mm]
u_2=a^2\left(\frac{(b^2v_2+1)v_3}{(v_3+b^2)v_1}\right)^2,& \ \ \ v_2=a^2\left(\frac{(b^2u_2+1)u_3}{(u_3+b^2)u_1}\right)^2,\\[3mm]
u_3=a^2\left(\frac{(b^2v_3+1)v_1}{(v_2+b^2)v_3}\right)^2,& \ \ \ v_3=a^2\left(\frac{(b^2u_3+1)u_1}{(u_2+b^2)u_3}\right)^2.
\end{array}\right.
\end{equation}

For sake of the simplicity, we are looking for a solution of the form $\mathbf{u}=(u,u,u)$.
We then obtain from \eqref{per} that
$$f_{a,b,2}(f_{a,b,2}(u))=u.$$
It is clear that any solution $u$ of the equation
$f_{a,b,2}(f_{a,b,2}(u))=u$ should be a prefect square of some
$p$-adic number $x$, i.e., $u=x^2$. So, we have to solve the
following equation with respect to $x$
$$
g_{a,b,2}(g_{a,b,2}(x))=x.
$$

Note that any solution of the equation $g_{a,b,2}(u)=u$ is also
solution of the equation $g_{a,b,2}(g_{a,b,2}(u))=u$. However, we
are looking for only 2-period solutions. Hence, we have to solve
the equation $ \frac{g_{a,b,2}(g_{a,b,2}(u))-u}{g_{a,b,2}(u)-u}=0.
$ This is equivalent to the following equation
\begin{equation}\label{kvx}
b^2(a^2b^2+1)x^2+a(b^4-1)x+b^2(b^2+a^2)=0.
\end{equation}

The discriminant of the quadratic equation  \eqref{kvx} is
$$
\D(a,b)=a^2(b^4-1)^2-4b^4(a^2b^2+1)(b^2+a^2).
$$
Since $a\equiv1(\operatorname{mod }p)$ and $b\equiv1(\operatorname{mod }p)$, we get that
\begin{equation}\label{normdelte}
\D(1,1)\equiv-16(\operatorname{mod }p).
\end{equation}
Thus, $\sqrt{\D(a,b)}$ exists in $\bq_p$ if and only if $\sqrt{-1}$ exists in $\bq_p$ or
equivalently $p\equiv1(\operatorname{mod }4)$. In this case, the quadratic equation \eqref{kvx}
has two solutions
$$
x_{\pm}=\frac{-a(b^4-1)\pm\sqrt{\D(a,b)}}{2b^2(a^2b^2+1)}.
$$
Due to Theorem \ref{pro1}, there exists a 2-periodic $p$-adic
quasi Gibbs measures associated with
$\mathbf{u}_{\pm}=(x_{\pm}^2,x_{\pm}^2,x_{\pm}^2).$ This completes
the proof.
\end{proof}

\section{uniqueness of $p$-adic Gibbs Measure}

In the previous section we have shown that if $p\geq 3$, then  the
set $\cg(H)$ is not empty for the $p$-adic Ising-Vannimenus model
on a Cayley tree of order $k\geq1$. In this section we are going
to  study the uniqueness of $p$-adic Gibbs measure for this model.

\begin{thm}\label{uniqGm}
Let $p\geq3$ and $|J|_p\leq\frac{1}{p}$,
$0<|J_{1}|_{p}\leq\frac{1}{p}$. Assume that $\mu_{\mathbf{u}}$ be
a $p$-adic Gibbs measure associated with the function
$\mathbf{u}=\{\mathbf{u}_{xy}\}_{\langle x,y\rangle\in L}$. Then
$\mu_{\mathbf{u}}$ is a translation-invariant if one of the
following conditions holds:

\begin{itemize}
\item[$(i)$] $u_{xy,1}=u_{xy,2}$ for all $\langle x,y\rangle\in L;$
\item[$(ii)$] $u_{xy,2}=u_{xy,3}$ for all $\langle x,y\rangle\in L;$
\item[$(iii)$] $u_{xy,1}=u_{xy,3}$ for all $\langle x,y\rangle\in L$.
\end{itemize}
\end{thm}

\begin{proof} Due to Theorem \ref{pro1}, it is enough to show that the system of equations \eqref{canonic_u}
has a unique solution whenever either one of the conditions $(i)-(iii)$ is satisfied.  It follows from \eqref{canonic_u} that
\begin{equation}\label{syst1}
\left\{\begin{array}{ll}
u_{xy,1}=a^2\prod\limits_{z\in S(y)}\frac{b^2u_{yz,3}+1}{u_{yz,3}+b^2}\\
u_{xy,3}=a^2\prod\limits_{z\in S(y)}\frac{b^2u_{yz,1}+1}{u_{yz,1}+b^2}
\end{array}\right.
\ \ \ \ \mbox{if }\ \ \ u_{xy,1}=u_{xy,2}\ \ \ \forall\langle{x,y}\rangle\in L
\end{equation}
\begin{equation}\label{syst2}
\left\{\begin{array}{ll}
u_{xy,1}=a^2\prod\limits_{z\in S(y)}\frac{b^2u_{yz,1}+1}{u_{yz,1}+b^2}\\
u_{xy,2}=a^2\prod\limits_{z\in S(y)}\frac{b^2u_{yz,2}+1}{u_{yz,2}+b^2}
\end{array}\right.
\ \ \ \ \mbox{if }\ \ \ u_{xy,1}=u_{xy,3}\ \ \ \forall\langle{x,y}\rangle\in L
\end{equation}
\begin{equation}\label{syst3}
\left\{\begin{array}{ll}
u_{xy,1}=a^2\prod\limits_{z\in S(y)}\frac{b^2u_{yz,1}+1}{u_{yz,1}+b^2}\\
u_{xy,1}=u_{xy,2}=u_{xy,3}
\end{array}\right.
\ \ \ \ \mbox{if }\ \ \ u_{xy,2}=u_{xy,3}\ \ \ \forall\langle{x,y}\rangle\in L
\end{equation}

Let us consider the case $(i)$, i.e., $u_{xy,1}=u_{xy,2}$ for all $\langle x,y\rangle\in L$.
We can apply the same method for the rest cases.

The main idea is to show that the mapping $G:\ce_p^2\to\ce_p^2$, $G=(G_1,G_2)$ defined as
below is contraction with respect to the $p$-adic norm  $\| \textbf{u}\|_p=\max\{|u_1|_p,|u_2|_p\}$
in the set $\ce_p^2$
$$
G_1(\textbf{u})=\frac{b^2u_2+a^{-2}}{u_2+a^{-2}b^2},\quad G_2(\textbf{u})=\frac{b^2u_1+a^{-2}}{u_1+a^{-2}b^2},\quad \textbf{u}=(u_1,u_2)\in\ce_p^2.
$$

One has for any $\textbf{u},\textbf{v}\in\ce_p^2$ that
\begin{equation}\label{F_i}
\left|G_i(\textbf{u})-G_i(\textbf{v})\right|_p=\frac{\left|(b^4-1)(u_j-v_j)\right|_p}{\left|a^2(u_j+a^{-2}b^2)(v_j+a^{-2}b^2)\right|_p}
=|b-1|_p\cdot|u_j-v_j|_p\leq\frac{1}{p}\|\textbf{u}-\textbf{v}\|_p
\end{equation}
where $i,j\in\{1,2\}$ and $i\neq j$. Therefore, $
\| G(\textbf{u})-G(\textbf{v})\|_p\leq\| \textbf{u}-\textbf{v}\|_p.
$, i.e., $G:\ce_p^2\to\ce_p^2$ is a contraction mapping.

Let $\{\textbf{u}^{(r)}\}_{r\in\bn}$ and $\{\textbf{v}^{(r)}\}_{r\in\bn}$ be sequences in $\ce_p^2$.
For any $n\geq1$, we have that
$$
\prod\limits_{s=1}^nG_i(\textbf{u}^{(s)})-\prod\limits_{s=1}^nG_i(\textbf{v}^{(s)})=
\sum\limits_{s=1}^n\left(\prod\limits_{l>s}G_i(\textbf{u}^{(l)})\left(G_i(\textbf{u}^{(s)})-
G_i(\textbf{v}^{(s)})\right)\prod\limits_{m<s}G_i(\textbf{v}^{(m)})\right)
$$
By means of Lemma \ref{epproperty} and \eqref{F_i}, we can find that
\begin{equation}\label{normF_i}
\left|\prod\limits_{s=1}^nG_i(\textbf{u}^{(s)})-\prod\limits_{s=1}^nG_i(\textbf{v}^{(s)})\right|_p\leq\frac{1}{p}\max\limits_{1\leq s\leq n}\left\{\Vert \textbf{u}^{(s)}-\textbf{v}^{(s)}\Vert_p\right\}.
\end{equation}

We define the mapping $\cf:\ce_p^2\to\ce_p^2,$ $\cf=(\cf_1,\cf_2)$ as follows
$$
\cf_1(\textbf{u}_{xy})=\prod\limits_{z\in S(y)}G_1(\textbf{u}_{yz}),\quad \cf_2(\textbf{u}_{xy})=
\prod\limits_{z\in S(y)}G_2(\textbf{u}_{yz}), \quad \langle x,y\rangle\in L
$$

Then by means of \eqref{normF_i}, one has for all $\langle x,y\rangle\in L$ that
\begin{equation}\label{fixxx}
\left\|\cf(\textbf{u}_{xy})-\cf(\textbf{v}_{xy})\right\|_p\leq\frac{1}{p}\max\limits_{z\in S(y)}\Vert \textbf{u}_{yz}-\textbf{v}_{yz}\Vert_p.
\end{equation}
Let $\mathbf{u}=\{\mathbf{u}_{xy}\}$ and $\mathbf{v}=\{\mathbf{v}_{xy}\}$
be solutions of \eqref{syst1}. Then for any $n\geq1$ from \eqref{fixxx} one gets
$$
\left\|\cf(\textbf{u}_{xy})-\cf(\textbf{v}_{xy})\right\|_p\leq\frac{1}{p^n},\ \ \mbox{for all }\langle x,y\rangle\in L
$$
Since arbitrarily of $n$ we obtain $\mathbf{u}=\textbf{v}$. Thus, we have shown that \eqref{syst1} has no more one solution.
If $u_{xy,i}=u_i,\ i=1,3$ for all $\langle x,y\rangle\in L$ we get
$\cf\equiv F$. From compactness of $\ce_p^2$ the function $F$ (respectively function
$\cf$) has a unique fixed point on $\ce_p^2$. This means that $\mu_{\mathbf{u}}$ is a translation-invariant.
\end{proof}

From this theorem and Theorem \ref{unique(k,p)} we immediately
obtain the following

\begin{cor} \label{unique(k,p)}
Let $p\geq3$ and $|J|_p\leq\frac{1}{p}$,
$0<|J_{1}|_{p}\leq\frac{1}{p}$. Then there exists a unique translation-invariant
$p$-adic Gibbs
measure corresponding to the model \eqref{ham1} on the Cayley tree
of order $k\geq1$.
\end{cor}

\begin{rem} The proved Theorem \ref{uniqGm} and Corollary \ref{unique(k,p)} partially confirm the conjecture formulated in
\cite{MFDMAH}.
\end{rem}

\begin{rem}
Note that the uniqueness of $p$-adic Gibbs measures for the
general nearest neighbor interaction models (which are called
$\l$-models) on arbitrary Cayley tree has been proved in
\cite{KM,MA3,MFDM}.
\end{rem}

From the proofs of Theorem \ref{uniqGm} and \ref{unique(k,p)}
and analyzing the equation \eqref{canonic_u} we may formulate the
following conjecture.
\begin{conj}
Let $p\geq3$ and $|J|_p\leq\frac{1}{p}$,
$0<|J_{1}|_{p}\leq\frac{1}{p}$. Then there exists a unique $p$-adic
Gibbs measure corresponding to the model \eqref{ham1} on the Cayley
tree of order $k\geq1$.
\end{conj}

\section{The phase transitions}

In this section, we establish the existence of the phase
transition for the $p$-adic Ising-Vannimenus model \eqref{ham1} on
a Cayley tree of order $k$. We need the following auxiliary
result.

\begin{prop}\label{recZ}
Let $\textbf{u}=(u_1,u_2,u_3)$ be any solution of the system of equations \eqref{tru}
and $Z_{n}^{(\textbf{u})}$ be a partition function \eqref{partition} associated with
$\textbf{u}$. Then one has that
$$
Z_{n}^{(u)}=\left(\frac{a(b^2u_3+1)}{bu_3}\right)^{k|V_{n-1}|}.
$$
\end{prop}
\begin{proof} We are going to provide some recurrence formula for the partition function
$Z_{n}^{(u)}$. For the given configuration $\s\in\Om_{V_n}$ and $n\geq1$, we denote
\begin{equation}\label{U_xy}
U^{(n)}_{xy,\s}=\left\{\begin{array}{ll}
1,&\ \mbox{if }\ \langle{x,y}\rangle\in\cn_{1,n}(\s)\\[3mm]
\frac{a^2}{u_{xy,3}},&\ \mbox{if }\ \langle{x,y}\rangle\in\cn_{2,n}(\s)\\[3mm]
\frac{a^2}{u_{xy,1}},&\ \mbox{if }\ \langle{x,y}\rangle\in\cn_{3,n}(\s)\\[3mm]
\frac{u_{xy,2}}{u_{xy,1}},&\ \mbox{if }\ \langle{x,y}\rangle\in\cn_{4,n}(\s)
\end{array}\right.
\end{equation}
There exists some function $D(x,y)\in\bq_p$ such that
\begin{equation}\label{D(x,y)}
\prod\limits_{z\in S(y)}\sum_{\w(z)\in\{\pm1\}}\exp_p\left[\w(z)(J\s(y)+J_1\s(x))\right]
U^{(n)}_{yz,\s\vee\w}=U^{(n-1)}_{xy,\s}D(x,y)
\end{equation}
where $x\in W_{n-2},\ y\in S(x)$ and $\s\in\Om_{V_{n-1}}$. It follows from the last equality that
$$
\prod\limits_{x\in W_{n-2}\atop{y\in S(x)}}D(x,y)U^{(n-1)}_{xy,\s}=
\prod\limits_{x\in W_{n-2}\atop{y\in S(x)}}
\prod\limits_{z\in S(y)}\sum_{\w(z)\in\{\pm1\}}\exp_p\left[\w(z)(J\s(y)+J_1\s(x))\right]U^{(n)}_{yz,\s\vee\w}
$$

Let $U_{n-1}=\prod\limits_{x\in W_{n-2}\atop{y\in S(x)}}D(x,y)$. By multiplying $\exp_p(H_{n-1}(\s))$
both sides of last equality, we obtain that
\begin{equation*}
\begin{split}
U_{n-1}&\exp_p(H_{n-1}(\s))\prod\limits_{x\in W_{n-2}\atop{y\in S(x)}}U^{(n-1)}_{xy,\s}=\\
&=\exp_p(H_{n-1}(\s))\prod\limits_{x\in W_{n-2}\atop{y\in S(x)}}
\prod\limits_{z\in S(y)}\sum_{\w(z)\in\{\pm1\}}\exp_p\left[\w(z)(J\s(y)+J_1\s(x))\right]U^{(n)}_{yz,\s\vee\w}
\end{split}
\end{equation*}

It follows from \eqref{U_xy} and \eqref{mu_u} that
\begin{equation}\label{canonicM}
U_{n-1}Z^{(u)}_{n-1}\mu_{u}^{(n-1)}(\s)=
Z^{(u)}_{n}\sum\limits_\w\mu_{u}^{(n)}(\s\vee\w).
\end{equation}

Since $\mu_{u}^{(n)}$ is a probability measure, i.e. for each
$n\in\bn$, one has
\[
\underset{\sigma\in
\Om_{V_{n-1}}}{\sum}\mu_{u}^{(n-1)}(\sigma)=\underset{\sigma\in
\Om_{V_{n-1}}}{\sum}\underset{\w\in\Om_{W_{n}}}{\sum}\mu_u^{(n)}(\sigma\vee\w)=1,
\]
we then find from \eqref{canonicM} that
\begin{equation}\label{canonic21}
Z^{(u)}_{n}=U_{n-1}Z^{(u)}_{n-1}.
\end{equation}

Now assume that $\textbf{u}_{xy}=(u_1,u_2,u_3)$ for any
$\langle{x,y}\rangle\in L$. It follows from \eqref{D(x,y)} with
$\s(x)=\s(y)=1$ that
$$
D(x,y)=\prod\limits_{z\in S(y)}(ab+(ab)^{-1}a^2u_3^{-1})=\left(\frac{a(b^2u_3+1)}{bu_3}\right)^k.
$$
Consequently, we get form \eqref{canonic21} that
$$
Z_n^{(u)}=\prod\limits_{\langle{x,y}\rangle\in L_{n-1}}\left(\frac{a(b^2u_3+1)}{bu_3}\right)^k
=\left(\frac{a(b^2u_3+1)}{bu_3}\right)^{k|V_{n-1}|}.
$$
\end{proof}

\begin{thm}\label{phtr}
Let $p\geq3$, $|J|_{p}\leq\frac{1}{p}$,
$0<\left|J_{1}\right|_{p}\leq\frac{1}{p}$. Assume that one of the
following conditions is satisfied:
\begin{enumerate}
\item[(i)] $(k,p)=1$ and $\frac{p-1}{(k,p-1)}$ is even;
\\
\item[(ii)]
 $k$ is odd and $\max\{|k|_p, |a-1|_p\}<|b-1|_p$;
\end{enumerate}
 then there
exists a phase transition for the $p$-adic Ising-Vannimenus model
\eqref{ham1}.
\end{thm}

\begin{proof} (i) Due to Theorem \ref{existenceGM}, we infer that there is a unique
$p$-adic Gibbs measure $\mu_0$ and  $d-$number of $p$-adic quasi Gibbs measures
$\mu_r$, $r=\overline{1,d}$ where $d=(k,p-1)$. Hence, in order to prove the
existence of the phase transition, we have to check
for boundedness/unboundedness of these measures.

As we already knew, the boundary functions corresponding to the
measures $\mu_r$, $r=\overline{0,d}$ are
$\textbf{u}_{(r)}=(u^{(r)}_1,u^{(r)}_2,u^{(r)}_3),\
r=\overline{0,d}$ in which $\textbf{u}_{(0)}\in\ce^3_p$ and
$\textbf{u}_{(r)}\in(-\ce_p)^3$, $ r=\overline{1,d}$. Due to
Proposition \ref{recZ}, we have
$$
Z_{n}^{(\textbf{u}_{(r)})}=\left(\frac{a(b^2u^{(r)}_3+1)}{bu^{(r)}_3}\right)^{k|V_{n-1}|},\quad
r=\overline{0,d}, \  n\in\bn.
$$
By using Lemma \ref{epproperty}, we can easily compute the $p$-adic norm
of $Z_{n}^{(\textbf{u}_{(r)})}$ as
\begin{equation}\label{normZ}
\left|Z_{n}^{(\textbf{u}_{(r)})}\right|_p=\left\{\begin{array}{ll}
1, & {if }\ \ r=0;\\
\leq p^{-k|V_{n-1}|}, & {if }\ \ r=\overline{1,d}.
\end{array}\right.
\end{equation}
It follows from \eqref{mu_u} and $|{u}^{(r)}_3|_p=1,\ r=\overline{0,d}$ that
$$
\left|\mu_{\textbf{u}_{(r)}}^{(n)}(\s)\right|_p=\frac{1}{\left|Z_{n}^{(\textbf{u}_{(r)})}\right|_p},\ \ \ \forall \  \s\in\Om_n.
$$
Therefore, from \eqref{normZ} one finds
$$
\left|\mu_{\textbf{u}_{(r)}}^{(n)}(\s)\right|_p=\left\{\begin{array}{ll}
1, & {if }\ \ r=0;\\
\geq p^{k|V_{n-1}|}, & {if }\ \ r=\overline{1,d}.
\end{array}\right.
$$
This means that the measure $\mu_0$ is bounded and the measures $\mu_{r}$
for all $r=\overline{1,d}$ are not bounded. Therefore, there exists a phase
transition for the $p$-adic Ising-Vannimenus model.

(ii) This case can be proceeded by the same argument as above.
This completes the proof.
\end{proof}

\begin{rem}
Theorems \ref{existenceGM}-\ref{perGM} and \ref{phtr} extend all
results of \cite{MFDMAH} to arbitrary order Cayley trees. However,
the question for an existence of the strong phase transition for
the $p$-adic Ising-Vannimenus model still remains to be open. We
can conditionally prove it by assuming the following conjecture.
\end{rem}

\begin{conj}\label{solutionz>1}
Let $p\geq3$ and $a,b\in\ce_p,\ b\neq1$. Let $\Delta$ be a set of
all solutions of the system  \eqref{tru}. If $(k,p)=1$ and
$\frac{p-1}{(k,p-1)}$ is even then
$\Delta\cap((\bq_p\setminus\bz_p)\times(-\ce_p)\times(-\ce_p))\neq\emptyset.$
\end{conj}

\begin{thm}\label{strpht}
Let $|J|_{p}\leq\frac{1}{p}$,
$0<\left|J_{1}\right|_{p}\leq\frac{1}{p}$. If Conjecture \ref{solutionz>1} holds
true then there exists a strong phase transition
for the $p$-adic Ising-Vannimenus model.
\end{thm}
\begin{proof}
Let $\textbf{u}=(u_1,u_2,u_3)\in\bq_p$ be a solution of the system
\eqref{tru} such that $|u_1|_p>1$. Due to Proposition
\ref{solutionnorm1}, we have that $|u_2|_p=|u_3|_p=1$ and
$|u_3+1|_p<1$. By means of the strong triangle inequality, it
follows from the first equation of the system \eqref{tru} that
$$
|u_1|_p=\left(\frac{\left|b^2u_3+1\right|_p}{\left|u_3+b^2\right|_p}\right)^k>\left|b^2u_3+1\right|_p^k.
$$
According to Proposition \ref{recZ}, from last inequality we
obtain that
\begin{equation}\label{normZ1}
\left|Z_{n}^{(\textbf{u})}\right|_p=\left(\frac{\left|a(b^2u_{3}+1)\right|_p}
{\left|bu_{3}\right|_p}\right)^{k|V_{n-1}|}<\left|u_1^{|V_{n-1}|}\right|_p.
\end{equation}

Let $\s\in\Om$ be a configuration such that $\s(x)=-1$ for all $x\in V\setminus\{x^{0}\}$.
It then follows from \eqref{normZ1} and \eqref{mu_u} that
$$
\left|\mu_{\textbf{u}}^{(n)}(\s)\right|_p=\frac{\left|u_1^{-|W_{n}|}\right|_p}
{\left|Z_{n}^{(u)}\right|_p}<\left|u_1^{-|V_{n}|}\right|_p.
$$
Since $|u_1|_p>1$ and $|V_n|=\frac{k(k^n-1)}{k-1}$, one has that
$$
\left|\mu_{\textbf{u}}^{(n)}(\s)\right|_p\to 0,\ \ \ \ \mbox{as }\ \ n\to\infty.
$$
Due to Theorem \ref{phtr}, we have for any $r=\overline{1,d}$ that
$$
\left|\mu_{u_{(r)}}^{(n)}(\s)\right|_p\to\infty,\ \ \ \ \mbox{as }\ \ n\to\infty.
$$
This means that there exists a strong phase transition for the $p$-adic Ising-Vannimenus model.
This completes the proof.
\end{proof}

\section{Conclusions}

In the present paper, we have considered the $p$-adic
Ising-Vanniminus  model on the Cayley tree of order $k$ ($k\geq
2$). This model is an analogue of the model studied in \cite{V}. A
new measure-theoretical approach is developed, in the $p$-adic
setting, to investigate such a model. We have constructed $p$-adic
quasi Gibbs measures via interacting functions. Such kind of
measures exist if the interacting functions satisfy
multi-dimensional recurrence equations. The existence of $p$-adic
quasi Gibbs measures are established by analyzing fixed points of
a multi-dimensional functional equation. It is shown that the
translation-invariant $p$-adic Gibbs measure is unique. This
partially confirms the conjecture formulated in \cite{MFDMAH}.
Finally, we are able to establish the existence of the phase
transition for the model depending on the order $k$ of the tree
and the prime $p$. One can observe that the values of the norms
(i.e. "absolute values" ), in the $p$-adic setting, are discrete,
therefore it is impossible to determine exact values of a critical
point for the existence of the phase transition. Moreover, in the
field of $p$-adic numbers there is no reasonable order compatible
with the usual order in the rational numbers, and this rises some
difficulties in the direction of determination of exact critical
points.
 Note that the methods used in the
paper are not valid in the real setting, since all of them based
on $p$-adic analysis and $p$-adic probability measures.

\section*{Acknowledgement} The work has been supported by the MOE grants ERGS13-024-0057 and ERGS13-025-0058.
The first named author (F.M.) acknowledges the IIUM grant EDW
B13-029-0914. Moreover, the third author (O.Kh.)  thanks Faculty
of Science of the International Islamic University Malaysia for
kind hospitality. Finally, the authors also would like to thank a
referee for his useful suggestions which allowed them to improve
the content of the paper.

\end{document}